\documentclass[a4paper,english,fontsize=10pt,parskip=half,abstracton]{scrartcl}
\usepackage{babel}
\usepackage[utf8]{inputenc}
\usepackage[T1]{fontenc}
\usepackage[a4paper,left=20mm,right=20mm,top=30mm,bottom=30mm]{geometry}
\usepackage{amsmath}
\usepackage{amsthm}
\usepackage{amssymb}
\usepackage{enumerate}
\usepackage{aliascnt}
\usepackage{tikz-cd}
\usepackage[bookmarks=false,
            pdftitle={Loewy lengths of centers of blocks II},
            pdfauthor={Burkhard Külshammer, Yoshihiro Otokita and Benjamin Sambale},
            pdfkeywords={},
            pdfstartview={FitH}]{hyperref}

\newtheorem{Theorem}{Theorem} 
\newaliascnt{Lemma}{Theorem}
\newtheorem{Lemma}[Lemma]{Lemma}
\aliascntresetthe{Lemma}
\newaliascnt{Proposition}{Theorem}
\newtheorem{Proposition}[Proposition]{Proposition}
\aliascntresetthe{Proposition}
\newaliascnt{Corollary}{Theorem}
\newtheorem{Corollary}[Corollary]{Corollary}
\aliascntresetthe{Corollary}

\numberwithin{equation}{section}

\allowdisplaybreaks[3]

\renewcommand{\phi}{\varphi}
\newcommand{\C}{\operatorname{C}}

\newcommand{\Z}{\operatorname{Z}}

\newcommand{\Br}{\operatorname{Br}}

\newcommand{\Ker}{\operatorname{Ker}}

\newcommand{\Span}{\operatorname{span}}

\newcommand{\Cl}{\operatorname{Cl}}

\mathchardef\ordinarycolon\mathcode`\:  
 \mathcode`\:=\string"8000
 \begingroup \catcode`\:=\active
   \gdef:{\mathrel{\mathop\ordinarycolon}}
 \endgroup

\title{Loewy lengths of centers of blocks II}
\author{Burkhard Külshammer\footnote{Institut für Mathematik, Friedrich-Schiller-Universität, 07743 Jena, Germany, 
\href{mailto:kuelshammer@uni-jena.de}{kuelshammer@uni-jena.de}}, Yoshihiro Otokita\footnote{Department of Mathematics and Informatics, Chiba University, Chiba–Shi, 263--8522, Japan, \href{mailto:otokita@chiba-u.jp}{otokita@chiba-u.jp}} \ and Benjamin Sambale\footnote{Fachbereich Mathematik, TU Kaiserslautern, 67653 Kaiserslautern, Germany, 
\href{mailto:sambale@mathematik.uni-kl.de}{sambale@mathematik.uni-kl.de}}}
\date{\today}

\begin{document}
\frenchspacing
\maketitle
\begin{abstract}\noindent
Let $ZB$ be the center of a $p$-block $B$ of a finite group with defect group $D$. 
We show that the Loewy length $LL(ZB)$ of $ZB$ is bounded by $\frac{|D|}{p}+p-1$ provided $D$ is not cyclic. 
If $D$ is non-abelian, we prove the stronger bound $LL(ZB)<\min\{p^{d-1},4p^{d-2}\}$ where $|D|=p^d$. 
Conversely, we classify the blocks $B$ with $LL(ZB)\ge\min\{p^{d-1},4p^{d-2}\}$.  
This extends some results previously obtained by the present authors. Moreover, we characterize blocks with uniserial center.
\end{abstract}


\section{Introduction}

The aim of this paper is to extend some results on Loewy lengths of centers of blocks obtained in \cite{KS2,Otokita}. In the following we will reuse some of the notation introduced in \cite{KS2}. In particular, $B$ is a block of a finite group $G$ with respect to an algebraically closed field $F$ of characteristic $p>0$. Moreover, let $D$ be a defect group of $B$. The second author has shown in \cite[Corollary~3.3]{Otokita} that the Loewy length of the center of $B$ is bounded by
\[LL(ZB)\le|D|-\frac{|D|}{\exp(D)}+1\]
where $\exp(D)$ is the exponent of $D$. It was already known to Okuyama~\cite{OkuyamaLL} that this bound is best possible if $D$ is cyclic. The first and the third author have given in \cite[Theorem~1]{KS2} the optimal bound $LL(ZB)\le LL(FD)$ for blocks with abelian defect groups. Our main result of the present paper establishes the following bound for blocks with non-abelian defect groups:
\[LL(ZB)<\min\{p^{d-1},4p^{d-2}\}\]
where $|D|=p^d$. As a consequence we obtain
\[LL(ZB)\le p^{d-1}+p-1\]
for all blocks with non-cyclic defect groups. 
It can be seen that this bound is optimal whenever $B$ is nilpotent and $D\cong C_{p^{d-1}}\times C_p$.

In the second part of the paper we show that $LL(ZB)$ depends more on $\exp(D)$ than on $|D|$. We prove for instance that $LL(ZB)\le d^2\exp(D)$ unless $d=0$.
Finally, we use the opportunity to improve a result of Willems~\cite{Willemsreptype} about blocks with uniserial center.

In addition to the notation used in the papers cited above, we introduce the following objects. 
Let $\Cl(G)$ be the set of conjugacy classes of $G$. A $p$-subgroup $P\le G$ is called a defect group of $K\in\Cl(G)$ if $P$ is a Sylow $p$-subgroup of $\C_G(x)$ for some $x\in K$. Let $\Cl_P(G)$ be the set of conjugacy classes with defect group $P$. Let $K^+:=\sum_{x\in K}x\in FG$ and 
\begin{align*}
I_P(G)&:=\langle K^+:K\in\Cl_P(G)\rangle\subseteq ZFG,\\
I_{\le P}(G)&:=\sum_{Q\le P}I_Q(G)\unlhd ZFG,\\
I_{<P}(G)&:=\sum_{Q<P}I_Q(G)\unlhd ZFG.
\end{align*} 

\section{Results}

We begin by restating a lemma of Passman~\cite[Lemma~2]{PassmanLem}. For the convenience of the reader we provide a (slightly easier) proof.

\begin{Lemma}[Passman]\label{Passlem}
Let $P$ be a central $p$-subgroup of $G$. Then $I_{\le P}(G)\cdot JZFG=I_{\le P}(G)\cdot JFP$.
\end{Lemma}
\begin{proof}
Let $K$ be a conjugacy class of $G$ with defect group $P$, and let $x\in K$. Then $P$ is the only Sylow $p$-subgroup of $\C_G(x)$, and the $p$-factor $u$ of $x$ centralizes $x$. Thus $u\in P$. Hence $u$ is the $p$-factor of every element in $K$, and $K=uK'$ where $K'$ is a $p$-regular conjugacy class of $G$ with defect group $P$. 
This shows that $I:=I_{\le P}(G)$ is a free $FP$-module with the $p$-regular class sums with defect group $P$ as an $FP$-basis. The canonical epimorphism $\nu:FG\to F[G/P]$ maps $I$ into $I_1(G/P)\subseteq SF[G/P]$. Thus $\nu(I\cdot JZFG)\subseteq SF[G/P]\cdot JZF[G/P]=0$. Hence $I\cdot JZFG\subseteq I\cdot JFP$. The other inclusion is trivial. 
\end{proof}

\begin{Lemma}\label{lem2}
Let $P\le G$ be a $p$-subgroup of order $p^n$. Then
\begin{enumerate}[(i)]
\item $I_{\le P}(G)\cdot JZFG^{LL(F\Z(P))}\subseteq I_{<P}(G)$.
\item $I_{\le P}(G)\cdot JZFG^{(p^{n+1}-1)/(p-1)}=0$.
\end{enumerate}
\end{Lemma}
\begin{proof}\hfill
\begin{enumerate}[(i)]
\item Let $\Br_P:ZFG\to ZF\C_G(P)$ be the Brauer homomorphism. Since $\Ker(\Br_P)\cap I_{\le P}(G)=I_{<P}(G)$, we need to show that $\Br_P(I_{\le P}(G)\cdot JZFG^{LL(F\Z(P))})=0$. By \autoref{Passlem} we have
\begin{align*}
\Br_P(I_{\le P}(G)\cdot JZFG^{LL(F\Z(P))})&\subseteq I_{\le\Z(P)}(\C_G(P))\cdot JZF\C_G(P)^{LL(F\Z(P))}\\
&=I_{\le\Z(P)}(\C_G(P))\cdot JF\Z(P)^{LL(F\Z(P))}=0.
\end{align*}

\item We argue by induction on $n$. The case $n=1$ follows from $I_1(G)\subseteq SFG$. Now suppose that the claim holds for $n-1$. Since $LL(F\Z(P))\le|P|=p^n$, (i) implies
\begin{align*}
I_{\le P}(G)\cdot JZFG^{(p^{n+1}-1)/(p-1)}&=I_{\le P}(G)\cdot JZFG^{p^n}JZFG^{(p^n-1)/(p-1)}\\
&\subseteq I_{<P}(G)\cdot JZFG^{(p^n-1)/(p-1)}\\
&=\sum_{Q<P}{I_{\le Q}(G)\cdot JZFG^{(p^n-1)/(p-1)}}=0.\qedhere
\end{align*}
\end{enumerate}
\end{proof}

Recall from \cite[Lemma~9]{KS2} the following group
\[W_{p^d}:=\langle x,y,z\mid x^{p^{d-2}}=y^p=z^p=[x,y]=[x,z]=1,\ [y,z]=x^{p^{d-3}}\rangle.\]
Note that $W_{p^d}$ is a central product of $C_{p^{d-2}}$ and an extraspecial group of order $p^3$.
Now we prove our main theorem which improves \cite[Theorem~12]{KS2}. 

\begin{Theorem}\label{nonabel}
Let $B$ be a block of $FG$ with non-abelian defect group $D$ of order $p^d$. Then one of the following holds
\begin{enumerate}[(i)]
\item $LL(ZB)< 3p^{d-2}$.
\item $p\ge 5$, $D\cong W_{p^d}$ and $LL(ZB)<4p^{d-2}$.
\end{enumerate} 
In any case we have \[LL(ZB)<\min\{p^{d-1},4p^{d-2}\}.\]
\end{Theorem}
\begin{proof}
By \cite[Proposition~15]{KS2}, we may assume that $p>2$. 
Since $D$ is non-abelian, $|D:\Z(D)|\ge p^2$ and $LL(F\Z(D))\le p^{d-2}$. Let $Q$ be a maximal subgroup of $D$. If $Q$ is cyclic, then $D\cong M_{p^n}$ and the claim follows from \cite[Proposition~10]{KS2}. Hence, we may assume that $Q$ is not cyclic. Then $LL(F\Z(Q))\le p^{d-2}+p-1$. Now setting $\lambda:=\frac{p^{d-1}-1}{p-1}$ it follows from \autoref{lem2} that
\begin{align*}
JZB^{2p^{d-2}+p-1+\lambda}&\subseteq 1_BJZFG^{2p^{d-2}+p-1+\lambda}\subseteq I_{\le D}(G)\cdot JZFG^{2p^{d-2}+p-1+\lambda}\\
&\subseteq I_{<D}(G)\cdot JZFG^{p^{d-2}+p-1+\lambda}=\sum_{Q<D}{I_{\le Q}(G)\cdot JZFG^{p^{d-2}+p-1+\lambda}}\\
&\subseteq \sum_{Q<D}{I_{<Q}(G)\cdot JZFG^{\lambda}}=0.
\end{align*}
Since $2p^{d-2}+p-1+\lambda\le 4p^{d-2}$, we are done in case $p\ge 5$ and $D\cong W_{p^d}$. 
If $p=3$ and $D\cong W_{p^d}$, then the claim follows from \cite[Lemma~11]{KS2}.
Now suppose that $D\not\cong W_{p^d}$. If $\Z(D)$ is cyclic of order $p^{d-2}$, then the claim follows from \cite[Lemma~9 and Proposition~10]{KS2}. Hence, suppose that $\Z(D)$ is non-cyclic or $\lvert\Z(D)\rvert<p^{d-2}$. Then $d\ge 4$ and $LL(F\Z(D))\le p^{d-3}+p-1$. The arguments above give $LL(ZB)\le p^{d-2}+p^{d-3}+2p-2+\lambda$, hence we are done whenever $p>3$.

In the following we assume that $p=3$. 
Here we have $LL(ZB)\le 3^{d-2}+3^{d-3}+4+\frac{1}{2}(3^{d-1}-1)$ and it suffices to handle the case $d=4$. By \cite[Theorem~3.2]{Otokita}, there exists a non-trivial $B$-subsection $(u,b)$ such that 
\[LL(ZB)\le (|\langle u\rangle|-1)LL(Z\overline{b})+1\] 
where $\overline{b}$ is the unique block of $F\C_G(u)/\langle u\rangle$ dominated by $b$. We may assume that $\overline{b}$ has defect group $\C_D(u)/\langle u\rangle$ (see \cite[Lemma~1.34]{habil}). If $u\notin\Z(D)$, we obtain $LL(ZB)<\lvert\C_D(u)\rvert\le 27$ as desired. Hence, let $u\in\Z(D)$. Then $D/\langle u\rangle$ is not cyclic. Moreover, by our assumption on $\Z(D)$, we have $|\langle u\rangle|=3$. 
Now it follows from \cite[Theorem~1, Proposition~10 and Lemma~11]{KS2} applied to $\overline{b}$ that 
\[LL(ZB)\le 2LL(Z\overline{b})+1\le 23<27.\qedhere\] 
\end{proof}

We do not expect that the bounds in \autoref{nonabel} are sharp.
In fact, we do not know if there are $p$-blocks $B$ with non-abelian defect groups of order $p^d$ such that $p>2$ and $LL(ZB)>p^{d-2}$. 
See also \autoref{coexp} below.

\begin{Corollary}\label{noncyc}
Let $B$ be a block of $FG$ with non-cyclic defect group of order $p^d$. Then 
\[LL(ZB)\le p^{d-1}+p-1.\]
\end{Corollary}
\begin{proof}
By \autoref{nonabel}, we may assume that $B$ has abelian defect group $D$. Then \cite[Theorem~1]{KS2} implies $LL(ZB)\le LL(FD)\le p^{d-1}+p-1$.
\end{proof}

We are now in a position to generalize \cite[Corollary~16]{KS2}.

\begin{Corollary}\label{cor}
Let $B$ be a block of $FG$ with defect group $D$ of order $p^d$ such that $LL(ZB)\ge \min\{p^{d-1},4p^{d-2}\}$. Then one of the following holds
\begin{enumerate}[(i)]
\item $D$ is cyclic.
\item $D\cong C_{p^{d-1}}\times C_p$.
\item $D\cong C_2\times C_2\times C_2$ and $B$ is nilpotent.
\end{enumerate}
\end{Corollary}
\begin{proof}
Again by \autoref{nonabel} we may assume that $D$ is abelian. By \cite[Corollary~16]{KS2}, we may assume that $p>2$. Suppose that $D$ is of type $(p^{a_1},\ldots,p^{a_s})$ such that $s\ge 3$. Then 
\begin{align*}
\min\{p^{d-1},4p^{d-2}\}&\le LL(ZB)=p^{a_1}+\ldots+p^{a_s}-s+1\\
&\le p^{a_1}+p^{a_2}+p^{a_3+\ldots+a_s}-2\le p^{d-2}+2(p-1).
\end{align*}
This clearly leads to a contradiction. Therefore, $s\le 2$ and the claim follows.
\end{proof}

In case (i) of \autoref{cor} it is known conversely that $LL(ZB)=\frac{p^d-1}{l(B)}+1>p^{d-1}$ (see \cite[Corollary~2.8]{KKS}).

Our next result gives a more precise bound by invoking the exponent of a defect group.

\begin{Theorem}\label{exp}
Let $B$ be a block of $FG$ with defect group $D$ of order $p^d>1$ and exponent $p^e$. Then
\[LL(ZB)\le \Bigl(\frac{d}{e}+1\Bigr)\Bigl(\frac{d}{2}+\frac{1}{p-1}\Bigr)(p^e-1).\]
In particular, $LL(ZB)\le d^2p^e$.
\end{Theorem}
\begin{proof}
Let $\alpha:=\lfloor d/e\rfloor$. Let $P\le D$ be abelian of order $p^{ie+j}$ with $0\le i\le \alpha$ and $0\le j<e$. 
If $P$ has type $(p^{a_1},\ldots,p^{a_r})$, then $a_i\le e$ for $i=1,\ldots,r$ and
\[LL(FP)=(p^{a_1}-1)+\ldots+(p^{a_r}-1)+1\le i(p^e-1)+p^j.\] 
Arguing as in \autoref{nonabel}, we obtain
\begin{align*}
LL(ZB)&\le
\sum_{i=0}^{\alpha}\sum_{j=0}^{e-1}{i(p^e-1)+p^j}=e(p^e-1)\Bigl(\sum_{i=0}^{\alpha}i\Bigr)+(\alpha+1)\frac{p^e-1}{p-1}\\
&=e(p^e-1)\frac{\alpha(\alpha+1)}{2}+(\alpha+1)\frac{p^e-1}{p-1}\\
&\le \Bigl(\frac{d}{e}+1\Bigr)\Bigl(\frac{d}{2}+\frac{1}{p-1}\Bigr)(p^e-1).
\end{align*}
This proves the first claim. For the second claim we note that 
\[\Bigl(\frac{d}{e}+1\Bigr)\Bigl(\frac{d}{2}+\frac{1}{p-1}\Bigr)\le (d+1)\Bigl(\frac{d}{2}+1\Bigr)\le d^2\]
unless $d\le 3$. In these small cases the claim follows from \autoref{nonabel} and \autoref{noncyc}. 
\end{proof}

If $2e>d$ and $p$ is large, then the bound in \autoref{exp} is approximately $dp^e$. 
The groups of the form $G=D=C_{p^e}\times\ldots\times C_{p^e}$ show that there is no bound of the form $LL(ZB)\le Cp^e$ where $C$ is an absolute constant.
A more careful argumentation in the proof above gives the stronger (but opaque) bound
\[LL(ZB)\le \alpha(p^e-1)\Bigl(\frac{e(\alpha-1)}{2}+\frac{1}{p-1}+d-\alpha e\Bigr)+\beta(p^e-1)+\frac{p^{d-\alpha e}-1}{p-1}+p^{d-2-\beta e}\]
for non-abelian defect groups where $\alpha:=\lfloor\frac{d-1}{e}\rfloor$ and $\beta:=\lfloor\frac{d-2}{e}\rfloor$. We omit the details.

In the next result we compute the Loewy length for $d=e+1$. 

\begin{Proposition}\label{coexp}
Let $B$ be a block of $FG$ with non-abelian defect group of order $p^d$ and exponent $p^{d-1}$. Then 
\[LL(ZB)\le\begin{cases}
2^{d-2}+1&\text{if }p=2,\\
p^{d-2}&\text{if }p>2
\end{cases}\]
and both bounds are optimal for every $d\ge 3$.
\end{Proposition}
\begin{proof}
Let $D$ be a defect group of $B$. 
If $p>2$, then $D\cong M_{p^d}$ and we have shown $LL(ZB)\le p^{d-2}$ in \cite[Proposition~10]{KS2}. Equality holds if and only if $B$ is nilpotent.

Therefore, we may assume $p=2$ in the following. 
The modular groups $M_{2^d}$ are still handled by \cite[Proposition~10]{KS2}.
Hence, it remains to consider the defect groups of maximal nilpotency class, i.\,e. $D\in\{D_{2^d},Q_{2^d},SD_{2^d}\}$. By \cite[Proposition~10]{KS2}, we may assume that $d\ge 4$. The isomorphism type of $ZB$ is uniquely determined by $D$ and the fusion system of $B$ (see \cite{Cabanes}). The possible cases are listed in \cite[Theorem~8.1]{habil}. If $B$ is nilpotent, \cite[Proposition~8]{KS2} gives $LL(ZB)=LL(ZFD)\le LL(FD')=2^{d-2}$. Moreover, in the case $D\cong D_{2^d}$ and $l(B)=3$ we have $LL(ZB)\le k(B)-l(B)+1=2^{d-2}+1$ by \cite[Proposition~2.2]{Otokita}. In the remaining cases we present $B$ by quivers with relations which were constructed originally by Erdmann~\cite{Erdmann}. We refer to \cite[Appendix B]{HolmHabil}. 

\begin{enumerate}[(i)]
\item $D\cong D_{2^d}$, $l(B)=2$:
\begin{center}
\begin{tabular}{lr}
$\begin{tikzcd}
\circ \arrow[out=135,in=225,loop,"\alpha",swap]\arrow[r,bend left,"\beta"] & \circ\arrow[in=315,loop,"\eta"]\arrow[l,bend left,"\gamma"]
\end{tikzcd}
$
&
$\begin{array}{c}
\beta\eta=\eta\gamma=\gamma\beta=\alpha^2=0,\\
\alpha\beta\gamma=\beta\gamma\alpha,\\
\eta^{2^{d-2}}=\gamma\alpha\beta.
\end{array}$
\end{tabular}
\end{center}

By \cite[Lemma 2.3.3]{HolmHabil}, we have
\[ZB=\Span\{1,\beta\gamma,\alpha\beta\gamma,\eta^i:i=1,\ldots,2^{d-2}\}.\]
It follows that $JZB^2=\langle\eta^2\rangle$ and $LL(ZB)=2^{d-2}+1$.

\item $D\cong Q_{2^d}$, $l(B)=2$:
Here \cite[Lemma~6]{Zimmermann} gives the isomorphism type of $ZB$ directly as a quotient of a polynomial ring 
\[ZB\cong F[U,Y,S,T]/(Y^{2^{d-2}+1},U^2-Y^{2^{d-2}},S^2,T^2,SY,SU,ST,UY,UT,YT).\]
It follows that $JZB^2=(Y^2)$ and again $LL(ZB)=2^{d-2}+1$.

\item $D\cong Q_{2^d}$, $l(B)=3$:
\begin{center}
\begin{tabular}{lr}
$\begin{tikzcd}
\circ\arrow[rr,shift left=.5ex,"\beta"]
\arrow[dr,shift left=.5ex,"\kappa"]&&\circ
\arrow[ll,shift left=.5ex,"\gamma"]
\arrow[dl,shift left=.5ex,"\delta"]\\[4ex]
&\circ\arrow[ul,shift left=.5ex,"\lambda"]
\arrow[ur,shift left=.5ex,"\eta"]
\end{tikzcd}
$
&
$\begin{array}{c}
\beta\delta=(\kappa\lambda)^{2^{d-2}-1}\kappa,\ \eta\gamma=(\lambda\kappa)^{2^{d-2}-1}\lambda,\\
\delta\lambda=\gamma\beta\gamma,\ \kappa\eta=\beta\gamma\beta,\ \lambda\beta=\eta\delta\eta,\\
\gamma\kappa=\delta\eta\delta,\ \gamma\beta\delta=\delta\eta\gamma=\lambda\kappa\eta=0.
\end{array}$
\end{tabular}
\end{center}
By \cite[Lemma~2.5.15]{HolmHabil},
\[ZB=\Span\{1,\beta\gamma+\gamma\beta,(\kappa\lambda)^i+(\lambda\kappa)^i,\delta\eta+\eta\delta,(\beta\gamma)^2,(\lambda\kappa)^{2^{d-2}},(\delta\eta)^2:i=1,\ldots,2^{d-2}-1\}.\]
We compute
\begin{align*}
(\beta\gamma+\gamma\beta)^2=(\beta\gamma)^2+(\gamma\beta)^2&=(\beta\gamma)^2+\delta\lambda\beta=(\beta\gamma)^2+(\delta\eta)^2,\\
(\beta\gamma+\gamma\beta)(\kappa\lambda+\lambda\kappa)&=\beta\gamma\kappa\lambda=\beta\delta\eta\delta\lambda=\beta\delta\eta\gamma\beta\gamma=0,\\
(\beta\gamma+\gamma\beta)(\delta\eta+\eta\delta)&=\gamma\beta\delta\eta=0,\\
(\beta\gamma+\gamma\beta)(\beta\gamma)^2&=(\beta\gamma)^3=\beta\gamma\beta\delta\lambda=0,\\
(\beta\gamma+\gamma\beta)(\lambda\kappa)^{2^{d-2}}&=0,\\
(\beta\gamma+\gamma\beta)(\delta\eta)^2&=\gamma\beta\delta\eta\delta\eta=0,\\
((\kappa\lambda)^{2^{d-2}-1}+(\lambda\kappa)^{2^{d-2}-1})(\kappa\lambda+\lambda\kappa)&=\kappa\eta\gamma+(\lambda\kappa)^{2^{d-2}}=(\beta\gamma)^2+(\lambda\kappa)^{2^{d-2}},\\
(\kappa\lambda+\lambda\kappa)(\delta\eta+\eta\delta)&=\lambda\kappa\eta\delta=0,\\
(\kappa\lambda+\lambda\kappa)(\beta\gamma)^2&=\kappa\lambda\beta\gamma\beta\gamma=\kappa\eta\delta\eta\gamma\beta\gamma=0,\\
(\kappa\lambda+\lambda\kappa)(\lambda\kappa)^{2^{d-2}}&=\lambda\kappa\eta\gamma\kappa=0,\\
(\kappa\lambda+\lambda\kappa)(\delta\eta)^2&=0,\\
(\delta\eta+\eta\delta)^2=(\delta\eta)^2+(\eta\delta)^2&=(\delta\eta)^2+\lambda\beta\delta=(\delta\eta)^2+(\lambda\kappa)^{2^{d-2}},\\
(\delta\eta+\eta\delta)(\beta\gamma)^2&=0,\\
(\delta\eta+\eta\delta)(\lambda\kappa)^{2^{d-2}}&=\eta\delta(\lambda\kappa)^{2^{d-2}}=\eta\delta\eta\gamma\kappa=0,\\
(\delta\eta+\eta\delta)(\delta\eta)^2&=\delta\lambda\beta\delta\eta=\gamma\beta\gamma\beta\delta\eta=0,\\
(\beta\gamma)^2(\beta\gamma)^2&=(\beta\gamma)^2(\lambda\kappa)^{2^{d-2}}=(\beta\gamma)^2(\delta\eta)^2=0,\\
(\lambda\kappa)^{2^{d-2}}(\lambda\kappa)^{2^{d-2}}&=(\lambda\kappa)^{2^{d-2}}(\delta\eta)^2=0,\\
(\delta\eta)^2(\delta\eta)^2&=\gamma\kappa\eta(\delta\eta)^2=\gamma\beta\gamma\beta(\delta\eta)^2=0.
\end{align*}
Hence, $JZB^2=\langle (\lambda\kappa)^2+(\kappa\lambda)^2,(\beta\gamma)^2+(\delta\eta)^2\rangle$ and $JZB^3=\langle(\lambda\kappa)^3+(\kappa\lambda)^3\rangle$. This implies $LL(ZB)=2^{d-2}+1$.

\item $D\cong SD_{2^d}$, $k(B)=2^{d-2}+3$ and $l(B)=2$:
\begin{center}
\begin{tabular}{lr}
$\begin{tikzcd}
\circ \arrow[out=135,in=225,loop,"\alpha",swap]\arrow[r,bend left,"\beta"] & \circ\arrow[in=315,loop,"\eta"]\arrow[l,bend left,"\gamma"]
\end{tikzcd}
$
&
$\begin{array}{c}
\gamma\beta=\eta\gamma=\beta\eta=0,\\
\alpha^2=\beta\gamma,\ \alpha\beta\gamma=\beta\gamma\alpha,\\
\eta^{2^{d-2}}=\gamma\alpha\beta.
\end{array}$
\end{tabular}
\end{center}

By \cite[Section~5.1]{HolmZimmermann}, we have
\[ZB=\Span\{1,\beta\gamma,\alpha\beta\gamma,\eta^i:i=1,\ldots,2^{d-2}\}.\]
As in (i) we obtain $JZB^2=\langle\eta^2\rangle$ and $LL(ZB)=2^{d-2}+1$.

\item $D\cong SD_{2^d}$, $k(B)=2^{d-2}+4$ and $l(B)=2$:
\begin{center}
\begin{tabular}{lr}
$\begin{tikzcd}
\circ \arrow[out=135,in=225,loop,"\alpha",swap]\arrow[r,bend left,"\beta"] & \circ\arrow[in=315,loop,"\eta"]\arrow[l,bend left,"\gamma"]
\end{tikzcd}
$
&
$\begin{array}{c}
\beta\eta=\alpha\beta\gamma\alpha\beta,\ \gamma\beta=\eta^{2^{d-2}-1},\\
\eta\gamma=\gamma\alpha\beta\gamma\alpha,\\
\beta\eta^2=\eta^2\gamma=\alpha^2=0.
\end{array}$
\end{tabular}
\end{center}

By \cite[Section~5.2.2]{HolmZimmermann}, we have
\[ZB=\Span\{1,\alpha\beta\gamma+\beta\gamma\alpha+\gamma\alpha\beta,\beta\gamma\alpha\beta\gamma,(\alpha\beta\gamma)^2,\eta^i,\eta+\alpha\beta\gamma\alpha:i=2,\ldots,2^{d-2}\}.\]
Since $(\alpha\beta\gamma)^2=\beta\eta\gamma=(\beta\gamma\alpha)^2$ and $(\gamma\alpha\beta)^2=\eta\gamma\beta=\eta^{2^{d-2}}$, it follows that
\[(\alpha\beta\gamma+\beta\gamma\alpha+\gamma\alpha\beta)^2=(\alpha\beta\gamma)^2+(\beta\gamma\alpha)^2+(\gamma\alpha\beta)^2=\eta^{2^{d-2}}.\]
Similarly,
\begin{align*}
(\alpha\beta\gamma+\beta\gamma\alpha+\gamma\alpha\beta)\beta\gamma\alpha\beta\gamma&=0,\\
(\alpha\beta\gamma+\beta\gamma\alpha+\gamma\alpha\beta)(\alpha\beta\gamma)^2&=0,\\
(\alpha\beta\gamma+\beta\gamma\alpha+\gamma\alpha\beta)\eta^2&=0,\\
(\alpha\beta\gamma+\beta\gamma\alpha+\gamma\alpha\beta)(\eta+\alpha\beta\gamma\alpha)&=0,\\
(\beta\gamma\alpha\beta\gamma)^2&=0,\\
\beta\gamma\alpha\beta\gamma(\alpha\beta\gamma)^2&=0,\\
\beta\gamma\alpha\beta\gamma\eta^2=\beta\gamma\alpha\beta\eta^2\gamma&=0,\\
\beta\gamma\alpha\beta\gamma(\eta+\alpha\beta\gamma\alpha)=\beta\gamma(\alpha\beta\gamma)^2\alpha&=0,\\
(\alpha\beta\gamma)^2(\alpha\beta\gamma)^2&=0,\\
(\alpha\beta\gamma)^2\eta^2&=0,\\
(\alpha\beta\gamma)^2(\eta+\alpha\beta\gamma\alpha)&=0,\\
\eta^2(\eta+\alpha\beta\gamma\alpha)&=\eta^3,\\
(\eta+\alpha\beta\gamma\alpha)^2&=\eta^2.
\end{align*}
Consequently, $JZB^2=\langle\eta^2\rangle$ and $LL(ZB)=2^{d-2}+1$.

\item $D\cong SD_{2^d}$, $l(B)=3$:
\begin{center}
\begin{tabular}{lr}
$\begin{tikzcd}
\circ\arrow[rr,shift left=.5ex,"\beta"]
\arrow[dr,shift left=.5ex,"\kappa"]&&\circ
\arrow[ll,shift left=.5ex,"\gamma"]
\arrow[dl,shift left=.5ex,"\delta"]\\[4ex]
&\circ\arrow[ul,shift left=.5ex,"\lambda"]
\arrow[ur,shift left=.5ex,"\eta"]
\end{tikzcd}
$
&
$\begin{array}{c}
\kappa\eta=\eta\gamma=\gamma\kappa=0,\ \delta\lambda=(\gamma\beta)^{2^{d-2}-1}\gamma,\\
\beta\delta=\kappa\lambda\kappa,\ \lambda\beta=\eta.
\end{array}$
\end{tabular}
\end{center}
From \cite[Lemma 2.4.16]{HolmHabil} we get
\[ZB=\Span\{1,(\beta\gamma)^i+(\gamma\beta)^i,\kappa\lambda+\lambda\kappa,(\beta\gamma)^{2^{d-2}},(\lambda\kappa)^2,\delta\eta:i=1,\ldots,2^{d-2}-1\}.\]
We compute
\begin{align*}
(\beta\gamma+\gamma\beta)((\beta\gamma)^{2^{d-2}-1}+(\gamma\beta)^{2^{d-2}-1})&=(\beta\gamma)^{2^{d-2}}+\delta\lambda\beta=(\beta\gamma)^{2^{d-2}}+\delta\eta,\\
(\beta\gamma+\gamma\beta)(\kappa\lambda+\lambda\kappa)&=\beta\gamma\kappa\lambda=0,\\
(\beta\gamma+\gamma\beta)(\beta\gamma)^{2^{d-2}}&=\beta\delta\lambda\beta\gamma=\kappa\lambda\kappa\eta\gamma=0,\\
(\beta\gamma+\gamma\beta)(\lambda\kappa)^2&=0,\\
(\beta\gamma+\gamma\beta)\delta\eta&=\gamma\beta\delta\eta=\gamma\kappa\lambda\kappa\eta=0,\\
(\kappa\lambda+\lambda\kappa)^2&=\beta\delta\lambda+(\lambda\kappa)^2=(\beta\gamma)^{2^{d-2}}+(\lambda\kappa)^2,\\
(\kappa\lambda+\lambda\kappa)(\beta\gamma)^{2^{d-2}}&=\kappa\lambda\beta\gamma(\beta\gamma)^{2^{d-2}-1}=\kappa\eta\gamma(\beta\gamma)^{2^{d-2}-1}=0,\\
(\kappa\lambda+\lambda\kappa)(\lambda\kappa)^2&=\lambda(\beta\gamma)^{2^{d-2}}\kappa=\eta\gamma(\beta\gamma)^{2^{d-2}-1}\kappa=0,\\
(\kappa\lambda+\lambda\kappa)\delta\eta&=0,\\
(\beta\gamma)^{2^{d-2}}(\beta\gamma)^{2^{d-2}}&=(\beta\gamma)^{2^{d-2}}(\lambda\kappa)^2=(\beta\gamma)^{2^{d-2}}\delta\eta=0,\\
(\lambda\kappa)^2(\lambda\kappa)^2&=(\lambda\kappa)^2\delta\eta=0,\\
(\delta\eta)^2&=\delta\lambda\beta\delta\eta=\delta\lambda\kappa\lambda\kappa\eta=0.
\end{align*}
Hence, $JZB^2=\langle (\beta\gamma)^2+(\gamma\beta)^2,(\kappa\lambda)^2+\delta\eta\rangle$ and $JZB^3=\langle(\beta\gamma)^3+(\gamma\beta)^3\rangle$. This implies $LL(ZB)=2^{d-2}+1$.\qedhere
\end{enumerate}
\end{proof}

It is interesting to note the difference between even and odd primes in \autoref{coexp}. For $p=2$, non-nilpotent blocks gives larger Loewy lengths while for $p>2$ the maximal Loewy length is only assumed for nilpotent blocks. 

Recall that a \emph{lower defect group} of a block $B$ of $FG$ is a $p$-subgroup $Q\le G$ such that 
\[I_{<Q}(G)1_B\ne I_{\le Q}(G)1_B.\] 
In this case $Q$ is conjugate to a subgroup of a defect group $D$ of $B$ and conversely $D$ is also a lower defect group since $1_B\in I_{\le D}(G)\setminus I_{<D}(G)$. 
It is clear that in the proofs of \autoref{nonabel} and \autoref{exp} it suffices to sum over the lower defect groups of $B$. In particular there exists a chain of lower defect groups $Q_1<\ldots <Q_n=D$ such that $LL(ZB)\le\sum_{i=1}^nLL(F\Z(Q_i))$. Unfortunately, it is hard to compute the lower defect groups of a given block.

The following proposition generalizes \cite[Theorem~1.5]{Willemsreptype}.

\begin{Proposition}
Let $B$ be a block of $FG$. Then $ZB$ is uniserial if and only if $B$ is nilpotent with cyclic defect groups.
\end{Proposition}
\begin{proof}
Suppose first that $ZB$ is uniserial. Then $ZB\cong F[X]/(X^n)$ for some $n\in\mathbb{N}$; in particular, $ZB$ is a symmetric $F$-algebra. Then \cite[Theorems 3 and 5]{OkuyamaTsushima} implies that $B$ is nilpotent with abelian defect group $D$. Thus, by a result of Broué and Puig~\cite{BrouePuig} (see also \cite{Kstructure}), $B$ is Morita equivalent to $FD$; in particular, $FD$ is also uniserial. Thus $D$ is cyclic.

Conversely, suppose that $B$ is nilpotent with cyclic defect group $D$. Then the Broué-Puig result mentioned above implies that $B$ is Morita equivalent of $FD$. Thus $ZB\cong ZFD=FD$. Since $FD$ is uniserial, the result follows.
\end{proof}

A similar proof shows that $ZB$ is isomorphic to the group algebra of the Klein four
group over an algebraically closed field of characteristic $2$ if and only if $B$ is nilpotent with
Klein four defect groups.

\section*{Acknowledgment}
The third author is supported by the German Research Foundation (project SA 2864/1-1) and the Daimler and Benz Foundation (project 32-08/13).

\end{document}